\documentclass[12 pt, psamsfonts]{amsart}
\usepackage{amsmath}
\usepackage{amsthm}
\usepackage{amsfonts}
\usepackage{amssymb}
\usepackage{tikz}
\usepackage{eucal}
\usepackage{graphicx}
\usepackage{multirow}
\usepackage[all,knot]{xy}
\xyoption{arc}

\newcommand{\CC}{\mathcal{C}}

\newcommand{\Rep}{{\rm Rep}}
\newcommand{\U}{{\rm U}}
\newcommand{\GL}{{\rm GL}}

\newcommand{\ot}{\otimes}
\newcommand{\B}{\mathcal{B}}

\newcommand{\e}{\mathbf{e}}

\newcommand{\bt}{\beta}

\newcommand{\Z}{\mathbb{Z}}

\newcommand{\Sn}{\mathfrak{S}}
\newcommand{\N}{\mathbb{N}}

\newcommand{\C}{\mathbb{C}}
\newcommand{\R}{\mathbb{R}}

\newcommand{\al}{\alpha}

\newcommand{\AGL}{{\rm AGL}}
\newcommand{\Aut}{{\rm Aut}}

\newtheorem{theorem}{Theorem}[section]
\newtheorem*{theorem*}{Theorem}
\newtheorem{lemma}[theorem]{Lemma}
\newtheorem{prop}[theorem]{Proposition}
\theoremstyle{definition}

\newtheorem{conj}[theorem]{Conjecture}

\theoremstyle{remark}
\newtheorem{remark}[theorem]{Remark}

\numberwithin{equation}{section}

\calclayout

\begin{document}

\title{Local representations of the loop braid group}
\author{Zolt\'an K\'ad\'ar$^{1}$, Paul Martin$^1$, Eric Rowell$^2$, and Zhenghan Wang$^{3}$}

\address{$^1$Department of Pure Mathematics\\
University of Leeds\\Leeds, LS2 9JT\\UK}
\email{zokadar@gmail.com, P.P.Martin@leeds.ac.uk}
\address{$^2$Department of Mathematics\\
    Texas A\&M University\\
    College Station, TX 77843-3368\\
    U.S.A.}
\email{rowell@math.tamu.edu}
\address{$^3$Microsoft Research, Station Q and Department of Mathematics\\ University of California\\ Santa Barbara, CA 93106\\U.S.A}
\email{zhenghwa@microsoft.com, zhenghwa@math.ucsb.edu}

\thanks{The first author is supported by EPSRC grant EP/I038683/1. The
  second author is partially supported by EPSRC grant EP/I038683/1.
The third author is partially supported by NSF grant DMS-1108725.  The fourth author, partially supported by NSF DMS 1108736, would like to thank the School of Mathematics and Department of Physics at University of Leeds for their hospitality during his visit, where this project began.}

\keywords{Loop braid group, braided vector space, TQFT}

\begin{abstract}
 We study representations of the loop braid group $LB_n$ from the perspective of extending representations
 of the braid group $\B_n$.  We also pursue a generalization of the braid/Hecke/Temperlely-Lieb paradigm---uniform finite dimensional quotient algebras of the loop braid group algebras.
\end{abstract}

\maketitle

\section{Introduction}

Non-abelian statistics of anyons in two spatial dimensions has attracted considerable attention largely due to topological quantum computation \cite{nayaktqc,wang10}.
Recently, non-abelian statistics is extended to statistics of point-like topological defects projectively \cite{bbcw2014}.  But an extension of non-abelian statistics of point-like excitations to three spatial dimensions is not possible.  However, loop or closed string excitations occur naturally in condensed matter physics and string theory.  Therefore, it is important to study statistics of extended objects in three spatial dimensions.

A systematical way to produce interesting and powerful representations of the braid group is via $(2+1)$-topological quantum field theories (TQFTs) \cite{wang10}.
Since the loop braid group is a motion group of sub-manifolds, we expect that interesting representations of the loop braid group could result from extended $(3+1)$-TQFTs.  But $(3+1)$-TQFTs are much harder to construct, and the largest known explicit class is the Crane-Yetter TQFTs based on pre-modular categories \cite{CY,ww12}.  The difficulty of constructing interesting representations of the loop braid group reflects the difficulty of constructing non-trivial $(3+1)$-TQFTs.  Potentially, given a pre-modular category $\mathcal{C}$, there are representations of all motion groups of sub-manifolds including the loop braid group associated to $\mathcal{C}$, but no explicit computation has been carried out for any non-trivial theory.  Hence, we will take a closely related, but different first step in the study of representations of the loop braid group.

The tower of group algebras of Artin's braid group $\B_n$, for $n\geq
1$ have topologically interesting quotients, such as the
Temperley-Lieb algebras \cite{jones86}, Hecke algebras \cite{jones87}
and BMW-algebras \cite{BW,murakami}.  Each of these algebras support a
Markov trace which then produces polynomial knot and link invariants.
Moreover, at roots of unity
many such
quotient algebras can be realized as endomorphism algebras in unitary
modular categories--the algebraic structure underlying certain
$(2+1)$-TQFTs \cite{Turaev94}.
These, in turn, describe the quantum symmetries of topological phases of matter in $2$ spatial dimensions \cite{wang13}.  The braid group representations associated with unitary modular categories would be physically realized as the motion of point-like particles in the disk $D^2$.
Our goal is to generalize this picture to topological systems in $3$ spatial dimensions with loop-like excitations.

The loop braid group $LB_n$ is the motion group of the $n$-component
oriented unlink inside the $3$-dimensional ball $D^3$
\cite{dahm,Gold,lin}.  It has appeared in other contexts as well:
it is
isomorphic to the braid-permutation group (see \cite{BWC}), the
welded-braid group (see \cite{FRR}) and the group of conjugating
automorphisms of
a certain
free group (see \cite{mcc}).  For an exploration of the structure as a semidirect product, see \cite{bard1}.  Very little is known about the linear representations of $LB_n$.  We investigate when a given representation of $\B_n$ may be extended to $LB_n$.  Some results in this direction are found in \cite{bard2} and \cite{Ver}.  For example, it is known that the faithful Lawrence-Krammer-Bigelow (LKB) representation of $\B_n$ does not extend to $LB_n$ for $n\geq 4$ except at degenerate values of the parameters (\cite{bard2}), but the Burau representation of $\B_n$ does extend.

It seems to be a rather hard problem to discover interesting finite-dimensional quotients of the tower of loop braid group algebras of $LB_n$.  Considering that the LKB representation appears in the BMW-algebra, we should not expect to simply extend known $\B_n$ quotients. Our approach is to consider extensions of $\B_n$ representations associated with solutions to the parameter-free Yang-Baxter equation.  This ensures that the quotient algebras are finite dimensional.  The main problem we study is when such representations extend.  One particular family of extendible representations are studied in some detail: the so-called affine group-type solutions.

The contents of the paper are as follows.  In Sec.~$2$, we recall a
presentation of the loop braid group.  In Sec.~$3$, we study
representations of the loop braid group from braided vector spaces,
and hence make the connection to Drinfeld doubles.
In Sec.~$4$, we initiate a general program to generalize the
braid/Hecke/Temperlely-Lieb paradigm---uniform
finite dimensional quotient algebras of the loop braid quotient
algebras, and report some preliminary analysis.
In particular we answer a question that has been open for some time,
raised in \cite[\S12.1]{Martin91}, about the structure of certain
`cubic' braid group
representations that lift to loop braid representations.

\section{The loop braid group and its relatives}

A fundamental theorem for loop braid groups is as follows.
\begin{theorem*}{\rm \cite{FRR}}
The \textbf{loop braid group} $LB_n$ is
isomorphic to
the abstract group generated by $2(n-1)$ generators
$\sigma_i$ and $s_i$ for $1\leq i\leq (n-1)$, satisfying the following three sets of relations:

The \textbf{braid relations}:
\begin{enumerate}
 \item[(B1)] $\sigma_i\sigma_{i+1}\sigma_i=\sigma_{i+1}\sigma_i\sigma_{i+1}$
 \item[(B2)] $\sigma_i\sigma_j=\sigma_j\sigma_i$ for $|i-j|>1$,
\end{enumerate}
the \textbf{symmetric group relations}:
\begin{enumerate}
 \item[(S1)] $s_is_{i+1}s_i=s_{i+1}s_is_{i+1}$
 \item[(S2)] $s_is_j=s_js_i$ for $|i-j|>1$,
 \item[(S3)] $s_i^2=1$
\end{enumerate}

and the \textbf{mixed relations}:
\begin{enumerate}
\item[(L0)] $\sigma_is_j=s_j\sigma_i$ for $|i-j|>1$
\item[(L1)] $s_is_{i+1}\sigma_i=\sigma_{i+1}s_is_{i+1}$
\item[(L2)] $\sigma_i\sigma_{i+1}s_i=s_{i+1}\sigma_i\sigma_{i+1}$
\end{enumerate}
\qed
\end{theorem*}

It is clear from this presentation that
the
subgroup generated by the $\sigma_i$ is Artin's braid group $\B_n$ while the $s_i$ generate the symmetric group $\mathfrak{S}_n$.
The loop braid group is a quotient of the \textbf{virtual braid group} $VB_n$ \cite{Ver} which satisfies all relations above except $(L2)$.

The relations $(L1)$ also hold if read backwards, \emph{i.e.} $s_{i+1}s_i\sigma_{i+1}=\sigma_i s_{i+1}s_i$, but $(L2)$ is not equivalent to its reverse:

\begin{enumerate}
\item[(L3)] $s_i\sigma_{i+1}\sigma_i=\sigma_{i+1}\sigma_{i}s_{i+1}$.
\end{enumerate}

However, in the transposed group $OLB_n$ (\emph{i.e.} the group that coincides with $LB_n$ as a set, but with the opposite multiplication $a* b=ba$) one has all relations as in $LB_n$ except $(L2)$ is replaced by $(L3)$.  Every group is isomorphic to its transposed group (via inversion) so we may freely work with either $LB_n$ or $OLB_n$.

We define the \textbf{symmetric loop braid group} $SLB_n$ to be $LB_n$ modulo the relations $(L3)$.  In particular we have surjections
$VB_n\twoheadrightarrow LB_n\twoheadrightarrow SLB_n$.

\section{$LB_n$ representations from braided vector spaces}
Several authors (see e.g. \cite{Ver}) have considered the question of extending representations of $\B_n$ to $LB_n$.  In this section we consider extending certain local representations of $\B_n$ (see \cite{RW}).

A \textbf{braided vector space} (BVS) $(V,c)$ is a solution $c\in GL(V^{\ot 2})$ to the Yang-Baxter equation:
$$(c\ot Id_V)(Id_V\ot c)(c\ot Id_V)=(Id_V\ot c)(c\ot Id_V)(Id_V\ot c).$$

Any BVS gives rise to a \textbf{local} representation $\rho^c$ of $\B_n$ via $\sigma_i\rightarrow Id_V^{\ot i-1}\ot c\ot Id_V^{\ot n-i-1}$.  An extension of $\rho^c$ to $LB_n$ or $OLB_n$ via $s_i\rightarrow Id_V^{\ot i-1}\ot S\ot Id_V^{\ot n-i-1}$ where $S\in\GL(V^{\ot 2})$ will also be called \textbf{local}.  The corresponding triple $(V,c,S)$ will be called a \textbf{loop braided vector space}.

A special case of local $\B_n$ representations through \emph{group-type} BVSs were introduced by Andruskiewitsch and Schneider \cite{AS}.  These play an important role in
their classification program for pointed finite-dimensional Hopf algebras.  We extend their definition slightly and say that a BVS $(V,c)$ is of \textbf{left group-type} (resp. \textbf{right group-type}) if there is an ordered basis $X:=[x_1,\ldots,x_n]$ of $V$ and $g_i\in GL(V)$ such that
$ c(x_i\ot z)=g_i(z)\ot x_i$ (resp. $c(z\ot x_j)=x_j\ot g_j(z)$) for all $i,j$ and $z\in V$.  There is a one-to-one correspondence between left and right group-type BVSs, since
the Yang-Baxter equation is invariant under $c\leftrightarrow c^{-1}$. Indeed, the inverse of $c(x_i\ot x_j)=g_i(x_j)\ot x_i$ is $c^{-1}(x_i\ot x_j)=x_j\ot g_j^{-1}(x_i)$, so that $(V,c)$ is a BVS of left group-type if and only if $c^{-1}$ is a BVS of right group-type.

\begin{lemma}\label{l:ybinbasis}
 Suppose that $(V,c)$ is a BVS of left group-type with respect to $X:=[x_1,\ldots,x_n]$ and corresponding $g_i$ defined on $X$ by $g_i(x_j):=\sum_{k=1}^n g_i^{j,k}x_k$.  If $g_i^{j,k}\neq 0$ then
$g_ig_j=g_kg_i$.
\end{lemma}

\begin{proof}
 we compute $$(c\ot I)(I\ot c)(c\ot I)(x_i\ot x_j\ot z)$$
and $$(I\ot c)(c\ot I)(I\ot
c)(x_i\ot x_j\ot z)$$ and compare the two sides.  This yields the equality:

$$\sum_{k=1}^ng_i^{j,k}g_ig_j(z)\ot x_k\ot x_i=\sum_{k=1}^ng_i^{j,k}g_kg_i(z)\ot
x_k\ot x_i.$$  Thus we see that if $g_i^{j,k}\neq 0$ then
$g_ig_j(z)=g_kg_i(z)$ for all $z$, and the result follows.
\end{proof}
The proof of Lemma \ref{l:ybinbasis} shows that the Yang-Baxter equation for $(V,c)$ of
left group type is equivalent to the matrix equation:

\begin{equation}\label{gteq}
 g_i^{j,k}g_ig_j=g_i^{j,k}g_kg_i \quad \mathrm{for\; all}\; i,j,k.
\end{equation}
A similar result may be derived for right group type BVSs.

If the $g_i$ act diagonally with respect to the basis $X$ so that $c(x_i\ot x_j)=q_{ij}(x_j\ot x_i)$ for some scalars $q_{ij}$ then we say
$(V,c)$ is of \textbf{diagonal type}.  More generally we will say that $(V,c)$ is \textbf{diagonalizable} if there exists a basis of $V$ with respect to which $(V,c)$ is a BVS of diagonal type.  We do not need to specify a handedness for diagonal type BVS, indeed we have:

\begin{lemma}\label{l:leftright}  A BVS $(V,c)$ is of both left and right group type if and only if  $(V,c)$ is diagonalizable.
\end{lemma}
\begin{proof}  If $c$ is of left group type with respect to $X$ and $g_i\in GL(V)$ and right group type with respect to $Y:=[y_1,\ldots,y_n]$ and $h_j\in GL(V)$ then $x_i\ot y_j$ is a basis for $V$, and $c(x_i\ot y_j)=g_i(y_j)\ot x_i=y_j\ot h_j(x_i)$.  This implies that the $g_i$ are simultaneously diagonalized in the basis $Y$ so that the $g_i$ pairwise commute.  Denote by $G$ the (abelian) group generated by the $g_i$ and let $g_i^{(j,k)}$ be the coefficient of $x_k$ in $g_i(x_j)$.  Since the $g_i$ pairwise commute, Lemma \ref{l:ybinbasis} shows that $g_i^{(j,k)}\not=0$ implies $g_j=g_k$.  Now note that the spaces $W_k:=\C\{x_j:g_j=g_k\}$ are $G$-invariant, and denote by $I_k:=\{j:x_j\in W_k\}$, so that the distinct $I_k$ partition $[n]$.  So choose a basis for each $W_k$ with respect to which each $g_i$ is diagonal, and denote the union of these bases by $Z$. It is clear that $g_i$ are diagonal with respect to the basis $Z$, but we
must check that $(V,c)$ is of group type with respect to this basis.  Let $z_k=\sum_{j\in I_k}  z_j^k x_j\in W_k\cap Z$.  Then
$$c(z_k\ot z_s)=\sum_{j\in I_k} z_j^k c(x_j\ot z_s)= \sum_{j\in I_k} g_j(z_s)\ot z_j^kx_j =q_{k,s}z_s\ot z_k$$
since all the $g_j$ with $j\in I_k$ are identical and so act by a common scalar $q_{k,s}$ on $z_s$.

The other direction is clear: diagonal type BVSs are of both left and right group type.
\end{proof}

BVSs of group type \emph{always} extend to loop BVSs, with left group-type BVSs giving representations of $OLB_n$ and right group-type BVSs giving representations of $LB_n$:

\begin{prop}\label{p:extension} Define $S(x_i\ot x_j):=x_j\ot x_i$. If $(V,c)$ is a BVS of left (resp. right) group-type then $(V,c,S)$ is a loop braided vector space.
\end{prop}

\begin{proof}  Define $\rho^c(s_i)=Id_V^{\ot i-1}\ot S\ot Id_V^{\ot n-i-1}$.  Relations $(B1),(B2),(S1),(S2),(S3)$ and $(L0)$ are immediate.
Since inversion gives an isomorphism from $LB_n$ to $OLB_n$ and produces a left group-type BVS from a right group-type BVS it suffices to check the relations $(L1)$ and $(L3)$ for $i=1$.
Firstly, $$\rho^c(s_1s_2\sigma_1)(x_i\ot x_j\ot x_k)=(x_k\ot g_i(x_j)\ot x_i)=\rho^c(\sigma_2 s_1 s_2)(x_i\ot x_j\ot x_k)$$ verifying $(L1)$.  Similarly,
$$\rho^c(\sigma_{2}\sigma_1s_{2})(x_i\ot x_j\ot x_k)=g_i(x_k)\ot g_i(x_j)\ot x_i=\rho^c(s_1\sigma_2\sigma_1)(x_i\ot x_j\ot x_k).$$ so we have $(L3)$.
\end{proof}

Suppose that $(V,c)$ is of left group-type, and we define $\rho^c(s_i)$ via $S$ as in the proof of Proposition \ref{p:extension}. Then $(L2)$ is satisfied if and only if the $g_i$ pairwise commute:
$$\rho^c(\sigma_1\sigma_2 s_1)(x_i\ot x_j\ot x_k)=g_jg_i(x_k)\ot j\ot i=g_ig_j(x_k)\ot x_j\ot x_i=\rho^c(s_2\sigma_1\sigma_2).$$
In particular, if $(V,c)$ is both of left and right group-type then $\rho^c$ extends to a representation of $SLB_n$.  More generally, we have:

\begin{prop}\label{p:diagonal extension}
Suppose that $(V,c)$ and $(V,S)$ are of diagonal type with respect to the (same) basis $X$ and $S^2=Id_{V^{\ot 2}}$.  Then $\rho^c$ extends to a representation of $SLB_n$ via $\rho^c(s_i)=Id_V^{\ot i-1}\ot S\ot Id_V^{\ot n-i-1}$.
\end{prop}
\begin{proof}  It suffices to check $(L1),(L2)$ and $(L3)$, which are straightforward calculations.
\end{proof}

Note, that in case $(V,c)$ is of group type (either right or left), $c$ takes a canonical form in terms of the basis
$X=[x_1,\dots,x_n]$ and in terms of that basis $S(x_i\otimes x_j)=\pm x_j\otimes x_i$
then $(V,c)$ is of diagonal type if $cS=Sc$. In this case the index of $\rho^c(B_n)$ in $\rho^c(LB_n)$ is finite.
\subsection{Affine group-type BVSs}

We are interested in local representations of $LB_n$ that detect symmetry, \emph{i.e.} that do not factor over $SLB_n$.
Fix $m\in\N$ and let $V$ be an $m$-dimensional vector space with basis $[x_1,\ldots,x_m]$.  For each $1\leq j\leq m$
define $h_j(x_i)=x_{\alpha i+\beta j}$ for some $\alpha,\beta\in\N$, where indices are taken modulo $m$.    We will determine sufficient conditions on $\alpha$ and $\beta$ so that $c(x_i\ot x_j):=x_j\ot h_j(x_i)$ gives $(V,c)$ the structure of a right BVS.  We will call these \textbf{affine} group-type BVSs.  For notational convenience we will identify $x_i$ with $i\pmod{m}$ and define $h_j(i)=\alpha i+\beta j$ where now $\alpha,\beta$ are integers modulo $m$, and denote $x_i\ot x_j$ by $(i,j)$.

The operator $h_j$ is invertible if and only if $\gcd(\alpha,m)=1$.  Since we are interested in finding BVSs that do not factor over $SLB_n$, we should look for non-diagonalizable affine BVSs.
By the proof of Lemma \ref{l:leftright} we see that a BVS corresponding to $\{h_j:1\leq j\leq m\}$ is diagonalizable if and only the $h_j$ pairwise commute.  Computing $h_ih_j(k)=h_jh_i(k)$ we see that this happens precisely when $(\alpha-1)\beta\equiv 0\pmod{m}$.  In particular we must assume that $\alpha\not\equiv 1\pmod{m}$ and $\beta\not\equiv 0\pmod{m}$.

By Proposition \ref{p:extension} as soon as we have determined values $\al$, $\bt$ so that $(V,c)$ is a (right) BVS we may extend $\rho^c$ to $LB_n$ by taking $S(x_i\otimes x_j)=x_j\ot x_i$.  Computing, we have:

$$\sigma_1\sigma_2\sigma_1(i,j,k)=(k,h_k(j),(h_k\circ h_i)(i))=\sigma_2\sigma_1\sigma_2(i,j,k)=(k,h_k(j),(h_{h_k(j)}\circ h_k)(i)).$$
Therefore we must have
$$(h_k\circ h_j)(i)=\al^2 i+\al\bt j+\bt k=(h_{h_k(j)}\circ h_k)(i)=\al^2 i+\al\bt(k+j)+\bt^2k,$$
that is, $\beta(\alpha+\beta)=\beta$.  One family of solutions corresponds to $\al+\bt=1$ so we set $t=\al$ and $\bt=(1-t)$.
We have proved:

\begin{theorem}
\label{p:asymmetric}
Let $m,t\in\N$ with $\gcd(m,t)=1$ and $t\not\equiv 1\pmod{m}$.  Then defining $h_j(x_i)=x_{ti+(1-t)j}$ and $S(x_i\ot x_j)=x_j\ot x_i$ (indices modulo $m$) on the basis $X:=[x_1,\ldots,x_m]$ gives rise to a loop braided vector space $(V,c,S)$ of $LB_n$ such that the corresponding $LB_n$ representation, $\varphi$, does not factor over $SLB_n$.
\end{theorem}

\begin{remark}
For $m$ prime, the family of loop braided vector spaces in Prop. \ref{p:asymmetric} are all possible non-diagonalizable affine BVSs, but for $m$ composite there are other solutions.  We will only focus on these solutions in the present work.
\end{remark}

It is clear from the construction that the representations $\varphi$ act by permutation on the standard basis vectors of $V^{\ot n}$.
By passing to the action on indices, we may identify the $\C$-representation $\varphi$ in Prop. \ref{p:asymmetric} with the following homomorphism $\rho_{m,t}:LB_n\rightarrow\GL_n(\Z_m)$ via

$$\rho_{m,t}(\sigma_i)=\begin{pmatrix}I_{i-1}&0&0\\0 & M &0\\0&0& I_{n-i-1}\end{pmatrix},\rho_{m,t}(s_i)= \begin{pmatrix}I_{i-1}&0&0\\0 & P &0\\0&0& I_{n-i-1}\end{pmatrix}$$
where $M=\begin{pmatrix} 0 & 1\\ t & 1-t\end{pmatrix}$ and $P=\begin{pmatrix}
                                                               0 &1\\1&0
                                                              \end{pmatrix}$
 with entries in $\Z_m$.  For later use, we point out that evaluating $\rho_{m,t}(\sigma_i)$ at $t=1$ gives $\rho_{m,t}(s_i)$.

We now investigate the images of these representations.

The restriction of $\rho_{m,t}$ to $\B_n$ may look familiar: it is nothing more than the (inverse of) the (unreduced) Burau representation, specialized at an integer $t$ with entries modulo $m$.  In light of \cite{Ver} it is not surprising that the Burau representation should admit an extension to $LB_n$ (although we caution the reader that \cite{Ver} may have a different composition convention than ours).  Observe that the row-sums of $\rho_{m,t}(\sigma_i)$ and $\rho_{m,t}(s_i)$ are $1$; therefore they are $n\times n$ (row)-stochastic matrices (modulo $m$).  In particular since the affine linear group $\AGL_{n-1}(\Z_m)$ is isomorphic to the group of $n\times n$ stochastic matrices modulo $m$ (see \cite{Po}, where $m$ prime is considered, but the proof is valid for any $m$), we see that the image of $\rho_{m,t}$ is a subgroup of $\AGL_{n-1}(\Z_m)$.  The question we wish to address is: When is $\rho_{m,t}:\,LB_n\rightarrow \AGL_{n-1}(\Z_m)$ surjective?

The group $\AGL_{n-1}(\Z_m)$ is the semidirect product of $(\Z_m)^{n-1}$ with $\GL_{n-1}(\Z_m)$ (with the obvious action).  The standard way to view $\AGL_{n-1}(\Z_m)$ is as the subgroup of $\GL_n(\Z_m)$ consisting of matrices of the form $\begin{pmatrix} A & v\\ 0 & 1\end{pmatrix}$ where $A\in\GL_{n-1}(\Z_m)$ and $v\in \Z_m^{n-1}$ (a column vector).  For economy of notation, we will denote these elements by $g(A,v)$.  In this notation the multiplication rule is: $$(A_1,v_1)(A_2,v_2)=(A_1A_2,A_1v_2+v_1).$$

To determine the conditions on $m,t$ so that $\rho_{m,t}$ is surjective, we need some additional notation and technical results.
\begin{itemize}
 \item For $i\neq j$, define $\Delta_{i,j}\in Mat(n)$ to be the matrix with $(i,j)$-entry equal to $1$ and all other entries zero.
 \item For $i\neq j$, define $E_{i,j}(\alpha)=I+\alpha\Delta_{i,j}$, \emph{i.e.} the elementary matrix corresponding to the row operation which adds $\alpha$ times the $j$th row to the $i$th row.
 \item Let $D(\alpha,i):=I+(\alpha-1)\Delta_{i,i}$ be the diagonal matrix with the $(i,i)$-entry equal to $\alpha$ and the remaining (diagonal) entries equal to $1$.
\end{itemize}

\begin{lemma}   Let $B=g(I,e_{i})\in\AGL_{n-1}(\Z_m)$, with $e_i\in(\Z_m)^{n-1}$ a standard basis vector. Then $\AGL_{n-1}(\Z_m)\subset \GL_n(\Z_m)$ is generated by $B$ and the following matrices:
\begin{enumerate}\label{l:generators}
\item[(a)] $E_{i,j}(1)$, all $1\leq i\neq j\leq n-1$ and
\item[(b)] $D(\alpha,1)$ all $\alpha\in\Z_m^\times$.
\end{enumerate}
\end{lemma}

\begin{proof}   Let $e_j\in(\Z_m)^{n-1}$ be an arbitrary standard basis vector and choose $A$ so that $Ae_i=e_j$.  Then $$g(A,0)g(I,e_i)g(A^{-1},0)=g(I,e_j).$$  Since the matrices $g(I,e_j)$ generate all elements of the form $g(I,b)$, $b\in(\Z_m)^{n-1}$, it is enough to show that matrices in (a) and (b) generate all matrices of the form $g(A,0)$ with $A\in\GL_{n-1}(\Z_m)$.

Since $[E_{i,j}(1)]^k=I+k\Delta_{i,j}=E_{i,j}(k)$ we see that we can obtain all elementary matrices corresponding to replacing row/column $i$ with a multiple of row/column $j$ plus row $i$.  Moreover, we may obtain all matrices that permute rows and all matrices of the form $D(\alpha,j)$ inductively from $D(-1,1)$ via:
$$\begin{pmatrix} 1&1\\0&1\end{pmatrix}\begin{pmatrix} -1&0\\0&1\end{pmatrix}\begin{pmatrix} 1&-1\\0&1\end{pmatrix}\begin{pmatrix} 1&1\\0&1\end{pmatrix}=\begin{pmatrix} 0&1\\1&0\end{pmatrix}.$$ Thus we obtain all elementary matrices in $\GL_{n-1}(\Z_m)$ as products of matrices as in (a) and (b).

Finally, observe that the $\gcd$ of the entries in any row/column of $A\in\GL_{n-1}(\Z_m)$ must be a unit in $\Z_m$.  Using elementary row/column operations (left/right multiplication by elementary matrices) we may transform $A$ into a matrix with the $(1,1)$ entry equal to $1$ and the remaining entries equal to zero.  It then follows by induction that every $A\in\GL_{n-1}(\Z_m)$ is a product of matrices as in (a) and (b), as required.
\end{proof}

\begin{prop}\label{p:affine image} Suppose that $t\in\Z$ is chosen so that $t$ and $(1-t)$ are units in $\Z_m$ and $\Z_m^\times=\langle t,-1\rangle$.  Then $\rho_{m,t}(LB_n)\cong \AGL_{n-1}(\Z_m)$.
\end{prop}

\begin{proof}  We proceed by induction on $n$.  For the case $n=2$ we must show that $M$ and $P$ as above generate $\AGL_{1}(\Z_m)$.  By taking the transpose of $M$ and $P$ followed by a change of basis we can transform these into our standard $\AGL_{1}(\Z_m)$ form as:
$\sigma=g(-t,t)=\begin{pmatrix} -t& t\\0&1 \end{pmatrix},\;s=g(-1,1)=\begin{pmatrix} -1 & 1\\0&1\end{pmatrix}.$

Now $g(-t,t)g(-1,1)=g(t,0)$, and $g(-1,1)g(t,0)g(-1,1)g(1/t,0)=g(1,1-t)$.  Since $(1-t)$ is invertible and $g(1,a)g(1,b)=g(1,a+b)$, we obtain all $g(1,a)$.  Since one of $t$ or $1-t$ is even, $2$ is a unit in $\Z_m$, with multiplicative inverse, say $i_2$. Now we compute $g(1,-i_2)g(-1,1)g(1,i_2)=g(-1,0)$.  Since $\Z_m^\times=\langle t,-1\rangle$ we obtain all $g(x,0)$ where $x\in\Z_m^\times$.  Therefore we have all $g(1,a)g(x,0)=g(x,a)\in\AGL_{1}(\Z_m)$.

Now we again take the transpose of $\rho_{m,t}(\sigma_i)$ and $\rho_{m,t}(s_i)$ for $1\leq i\leq n-1$ and then change to the ordered basis:
$[(1,\ldots,1),(0,1,\ldots,1),\ldots,(0,\ldots,0,1)]$, so that the generators have the form $g(A,a)$ with $A\in\GL_{n-1}(\Z_m)$ and $a\in(\Z_m)^{n-1}$.  By the induction hypothesis, the images of $\sigma_i,s_i$ for $1\leq i\leq n-2$ generate all matrices of the form $g(B,0)$ where $B\in\AGL_{n-2}(\Z_m)$.  That is, we have all $g(g(C,c),0)$ with $C\in\GL_{n-2}(\Z_m)$ and $c\in(\Z_m)^{n-2}$.  With respect to this basis the image of the generator $\sigma_{n-1}$ has the form $\Sigma_{n-1}(t):=g(J,te_{n-1})$ where $J=\begin{pmatrix} 1&0&\cdots&0\\ \vdots & \ddots & \cdots &\vdots\\ 0& \cdots & 1 &0\\ 0& \cdots & 1&-t\end{pmatrix}$, and the image of the generator $s_i$ is obtained by evaluating $\Sigma_{n-1}(t)$ at $t=1$.

We have now reduced to showing that $g(g(C,c),0)$ together with $\Sigma_{n-1}(t)$ and $\Sigma_{n-1}(1)$ generate all of $\AGL_{n-1}(\Z_m)$.  By Lemma \ref{l:generators} it suffices to obtain $g(I,e_{n-1})$ as well as all $g(E_{i,j}(1),0)$ for $1\leq i,j\leq n-1$ and $g(D(\alpha,1),0)$ for all $\alpha\in\Z_m^\times$.  Since $C\in\GL_{n-2}(\Z_m)$ and $c\in(\Z_m)^{n-2}$ can be chosen arbitrarily, we immediately obtain all $g(D(\alpha,1),0)$ as well as the $g(E_{i,j}(1),0)$ for $i\leq n-2$ and $1\leq j\leq n-1$.

Let $e_{n-1}$ denote the standard basis vector in $(\Z_m)^{n-1}$ and set $$T(t):=\Sigma_{n-1}(t)\Sigma_{n-1}(1)\Sigma_{n-1}(t)^{-1}\Sigma_{n-1}(1)=I+(1-t)(\Delta_{n-1,n-2}-\Delta_{n-1,n}).$$  We compute
$T(t)^k=I+k(1-t)(\Delta_{n-1,n-2}-\Delta_{n-1,n})$ and since $(1-t)$ is invertible modulo $m$ we may choose $k=(1-t)^{-1}$ to obtain $T(0)=I+(\Delta_{n-1,n-2}-\Delta_{n-1,n})$.  Now we compute: $$g(D(-1,n-2),0)T(0)g(D(-1,n-2),0)T(0)=g(I,-2e_{n-1}).$$
Since $-2$ is invertible modulo $m$, we may appeal to Lemma \ref{l:generators} to produce all elements of the form $g(I,b)$, once we obtain the remaining generators of $\GL_{n-1}(\Z_m)$.

 Thus it remains to produce $g(E_{n-1,j}(1),0)$ for all $1\leq j\leq n-2$.  For this we set $X=g((I+\sum_{j=1}^{n-2} a_i\Delta_{n-3,i})D(a_{n-2},n-2),0)$, that is, the $n\times n$ matrix with the $(n-2)$th row equal to $(a_1,\ldots,a_{n-2},0,0)$ and $X_{i,j}=\delta_{i,j}$ for $i\neq (n-2)$.  Notice that $X$ is of the form $g(g(C,0),0)$ with $C\in\GL_{n-2}(\Z_m)$, assuming that $a_{n-2}$ is invertible.  Setting $Z=X^{-1}\Sigma_{n-1}(1)X\Sigma_{n-1}(1)$ we find that the $(n-1)$th row of $Z$ has entries $(a_1,\ldots,a_{n-3},a_{n-2}-1,1,0)$ and $Z_{i,j}=\delta_{i,j}$ for $i\neq (n-1)$.  Specializing at appropriate values of $a_i$ (e.g. $a_{n-2}\in\{1,2\}$, $a_i\in\{0,1\}$ for $i<n-2$) we obtain all $g(E_{n-1,j}(1),0)$ for $1\leq j\leq n-2$.  Thus, by Lemma \ref{l:generators} we have completed the induction and the result follows.

\end{proof}

\begin{remark}
We conjecture that Prop. \ref{p:affine image} is sharp.

Clearly $\{\det(T):T\in\AGL_{n-1}(\Z_m)\}=\Z_m^\times$.
Since $\det(M)=-t$ and $\det(S)=-1$, the image of $\rho_{m,t}(LB_n)$ consists of matrices with determinant $\pm t^k$.  This shows if $\rho_{t,m}(LB_n)\cong\AGL_{n-1}(\Z_m)$ then $\langle t,-1\rangle = \Z_m^\times$.  In particular if $\Z_m^\times$ is not a cyclic group or the direct product of $\Z_2$ with a cyclic group $\Z_d$ then $\rho_{t,m}(LB_n)$ is a proper subgroup of $\AGL_{n-1}(\Z_m)$.  Clearly $t$ and $(1-t)$ can both be units only if $m$ is odd.  In this case, the group $\Z_m^\times\cong \Z_d\times\Z_2$ if only if $m=p^aq^b$ is a product of at most $2$ odd primes and $\gcd(p^a-p^{a-1},q^b-q^{b-1})=2$.
\end{remark}

\subsection{Relationship with Drinfeld doubles}
In \cite{GR} it is observed that a BVS $(V,c)$ with corresponding
operators $g_1,\ldots,g_n$ may be realized as a Yetter-Drinfeld module
over the group $G=\langle g_1,\ldots,g_n\rangle$.  When $G$ is finite,
these can be identified with objects in $\Rep(DG)$
(regarded as a braided fusion category)
where $DG$ is the Drinfeld double of the group $G$.

As a vector space $DG=G^\C\ot \C[G]$ where $G^\C$ is the Hopf algebra of functions on $G$ with basis $\delta_g(h)=\delta_{g,h}$ and $\C[G]$ is the (Hopf) group algebra.  The Hopf algebra structure on $DG$ is well-known.  For an account of the associated braid group representations (and further details) see \cite{ERW}.

The irreducible representations of $DG$ are labeled by pairs $(\overline{g},\chi)$ where $\overline{g}$ is a conjugacy class in $G$ and $\chi$ is the character of an irreducible representation of the centralizer of $g$ in $G$: $C_G(g)$.  The representation $\rho_{m,t}$ of Prop. \ref{p:asymmetric} can be obtained in this way.  We now describe this explicitly.

Let $m,t$ be positive integers with $\gcd(m,t)=1$ and $t\neq 1\pmod{m}$.  Let $\ell$ be the order of $t$ modulo $m$, and $\Z_m=\langle r\rangle$ be the cyclic group modulo $m$ with generator $r$.  The map $\tau(r)= r^t$ defines an automorphism of $\Z_m$, which generates a cyclic subgroup $\Z_\ell$ of $\Aut(\Z_m)$.  Therefore we may form the semidirect product $G=\Z_m\rtimes \Z_{\ell}$ via

$$
    srs^{-1} = r^t
$$
where $\langle s\rangle=\Z_\ell$.  Let us further assume that $\gcd(m,t-1)=1$.
It follows from the relations above that
$r^isr^{-i} =r^{i(1-t)}s$ for all $i$, and the conjugacy class of $s$
is $\{r^{i(1-t)}s: 0\leq i\leq m-1\}$. For notational convenience, let $q=r^{1-t}$ so that
$q$ has order $m$ and the conjugacy class of $s$ is $\{q^is: 0\leq i\leq m-1\}$.
Then $V=V_{(s,1)}$ has basis $\{q^i\mid 0\leq i\leq m-1\}$, a set of
coset representatives of $G/C_G(s)$. The action of the $R$-matrix of $DG$ $\check{R}$
on $V\ot V$ is (where $P$ denotes the usual transposition):
\begin{eqnarray*}
\check{R}(q^i\ot q^j) &=& P R (q^i\ot q^j)\\
    &=& P (\sum_{g\in G}\delta_g\ot g)(q^i\ot q^j)\\
   &=& P (q^i\ot q^{i(1-t)} s q^j)\\
   &=& P (q^i\ot q^{i(1-t)+jt})\\
   &=& q^{i(1-t)+jt}\ot q^i.
\end{eqnarray*}

Clearly we may identify $\check{R}$ with the $\Z$-linear operator on
$\Z_{m}\times \Z_{m}$ given by
$$
   (i,j)\mapsto ((1-t)i+tj,i).
$$
This is the transpose of the braided vector space described in Prop. \ref{p:asymmetric}.


\newcommand{\mpr}[1]{\begin{prop} #1 \end{prop}}
\newcommand{\mlem}[1]{\begin{lemma} #1 \end{lemma}}

\newcommand{\ignore}[1]{}

\newcommand{\braidrel}[2]{#1_{#2} #1_{#2+1} #1_{#2}
                             = #1_{#2+1}  #1_{#2} #1_{#2+1}}
\newcommand{\commrel}[3]{#1_{#2} #1_{#3} = #1_{#3} #1_{#2}
      \qquad (|#2 - #3 |>1)}
\newcommand{\kk}{{\mathbf k}}
\newcommand{\LB}{LB}  
\newcommand{\f}{{\mathbf f}}  
\newcommand{\Id}{Id}  
\newcommand{\beq}{\begin{equation}}
\newcommand{\eq}{\end{equation}}
\def\mat{ \left( \begin{array} }
\def\tam{ \end{array}  \right) }

\newcommand{\rhoo}{\tau}  
\newcommand{\rhoox}{\tau^{x}}  
\newcommand{\rhooq}{\tau^{}}  
\newcommand{\BB}{B} 
\newcommand{\Bt}{B^{\rhoo_N}} 
\newcommand{\Btt}{B^{3}} 


\section{Finite dimensional quotient algebras}

In order to study certain local and finite-dimensional representations $\rho$ of $LB_n$,
such as the BVS representations $\rho^c$ described in \S 3  above, we are interested in certain finite-dimensional
quotient algebras of the group algebra $\C[LB_n]$,
namely the algebras
\newcommand{\Lrho}{L^{\rho}}
$$
\Lrho_n \; := \; \C[LB_n] /\ker \rho  .
$$
In passing to the group algebra, we linearize.
Thus
$
\ker\rho = \{ x \in \C[LB_n] \; | \; \rho(x) = 0 \} .
$
This should be contrasted with the group representation version:
$
\ker_G\rho = \{ g \in LB_n \; | \; \rho(g) = 1 \} .
$
It can easily happen that $\ker_G\rho = \{ 1 \}$ but
$\ker\rho \neq \{ 0 \}$.
\footnote{Note also that while group representations
and group algebra representations are closed under
  tensor products, the linear kernel is not preserved in general.
}
This raises some questions. (1): What is a good presentation of $\Lrho_n$ for each $n$?
Can the kernel be described in closed form for all $n$?
(2): What are the irreducible representations of  $\Lrho_n$?

In this section, we first use an analogy to show why the answers
to these questions will be useful.
This analogy shows that the study of the
quotient algebras $\Lrho_n$
is of {\em intrinsic} interest.
Then we analyse these representations, and answer (2) in
certain cases.

\subsection{A braid group quotient analogy}\label{ss:-1}
\newcommand{\qq}{{\underline{q}}} 

Consider the ordinary braid group $\B_n$.
For each $N$, $V = \C^N$ and $q \in \C^*$ there is a
well-known BVS with $c=c_N$, where
in the case $N=2$:
\ignore{{
\[
c_2 = \;
-q^{-1} \mat{cccc} 1 \\ & 1-q^2 & -q  \\ & -q & 0 \\ &&& 1 \tam
= \;  \mat{cccc} -q^{-1} \\ & q-q^{-1} & 1  \\ & 1 & 0 \\ &&& -q^{-1} \tam
\]
or }}
\[
c_2 = \;
q^{} \mat{cccc} 1 \\ & 1-q^{-2} & -q^{-1}  \\ & -q^{-1} & 0 \\ &&& 1 \tam
= \;  \mat{cccc} q^{} \\ & (q-q^{-1} )& -1  \\ & -1 & 0 \\ &&& q^{} \tam
\]
We write $\rho_N$ for the representation $\rho^{c_N}$.

For $\qq = (q_1, q_2, ...)$ a tuple in $\C^*$ define
$\chi_\qq = \prod_i (\sigma_1 - q_i) \in \C[\B_n]$.
The Hecke algebra is
$
H_n(q) = \C[B_n] / I_q,
$
where $I_q$ is the ideal generated by
$\chi_{(q,-q^{-1})} = (\sigma_1 -q)(\sigma_1 +q^{-1}) \in \C[B_n]$
for some $q \in \C^*$.
Evidently $\rho_N$ factors through $H_n$, but it is not linearly
faithful for all $n$.
The quotients
$
H_n^N(q)
$
are defined by
$
H^N_n = \C[\B_n] / \ker\rho_N.
$

So what is a good presentation for $H^N_n$ for given $N$?
There is
an element $\f$ of $H_m$ for $m=N+1$ such that
\beq
\label{eq:HfH}
\ker \rho_N \; = \; H_n \f H_n
\eq
for all $n$ (with the kernel understood to be 0 for $n<m$).
To construct $\f$ for a given $N$,
recall that for each $m$
there is a nonzero element $\f_m$ of $H_m$ unique up to
scalars such that
$\sigma_i \f_m = \f_m \sigma_i = (-q^{-1}) \f_m$ for all $i$.
For example we can take
$\f_2 \; = U_1 $ where
$
U_i \; := \sigma_i -q
$
and $\f_3 = U_1 U_2 U_1 - U_1$.
We may take $\f = \f_{N+1}$ or any nonzero scalar multiple thereof.
That is, there is a single additional relation that characterises
$H^N_n$ as a quotient of $H_n$ for all $n$ and $q$, namely
$
\f_{N+1} = 0
$
\cite{Martin92}. Thus $H^2_n$ is the Temperley--Lieb algebra and so on.

\subsection{On localization}\label{ss:loc}
Given an algebra $A$, let $\Lambda(A)$ be the set of irreducible
representations up to isomorphism.
Another feature of the braid/Hecke/Temperley--Lieb paradigm is {\em localization}.

\newcommand{\modd}{{\mbox{mod}}}
\newcommand{\Fe}{F_e}
\newcommand{\Ge}{G_e}
\newcommand{\ad}[2]{\raisebox{.1cm}{$\longleftarrow$}
   \hspace{-.7cm} \raisebox{-.1cm}{$\longrightarrow$}
   \hspace{-.5cm} \raisebox{.321cm}{$#1$}
   \hspace{-.518cm} \raisebox{-.4321cm}{$#2$}
}    
Given an algebra $A$ and idempotent $e \in A$, then $eAe$ is also an
algebra (not a subalgebra) and the functors $\Ge,\Fe$: 
\beq \label{eq:FG1}
A-\!\!\modd \;\;\;
    \ad{\Ge}{\Fe}
       \;\;\; eAe-\!\!\modd
\eq
(`globalization' and `localization', respectively) given on modules by
\[
\Ge N = Ae \otimes_{eAe} N
\]
\[
\Fe M = eM
\]
are an adjoint pair.
Useful corollaries to this include the following:
\\ (LI) Let $L_i, L_j$ be distinct simple $A$-modules,
with $eL_i$ and $eL_j$ nonzero.
Then  
$eL_i$, $eL_j$ are distinct simple $eAe$-modules.
\\ (LII) If $L_i$ has composition multiplicity $m_i$ in $A$-module
$M$, and $eL_i \neq 0$,
then $eL_i$ has multiplicity $m_i$ in $eAe$-module $eM$:
\beq \label{eq:cm1}
[M : L_i ] \; = \; [eM : eL_i ]
\eq
(LIII) The set $\Lambda(A)$ of irreducible representations of
$A$ (up to isomorphism) is in bijection with the disjoint union of
those of $eAe$ and those of $A/AeA$:
\beq \label{eq:LamA1}
\Lambda(A) \; \cong \; \Lambda(eAe) \sqcup \Lambda(A/AeA)
\eq

The idea here is very general. Given an algebra $A$ to study, we find an
idempotent in it, then study $A$ by studying $eAe$ and $A/AeA$.
{\em In} general $eAe$ and $A/AeA$ are also unknown and
this subdivision does not help much.
But for $H^N_n$ we have an $e$ such that
\beq \label{eq:eHe00}
e H^N_n e \cong H^N_{n-N}  ,
\eq
so we can consider $eAe$ to be known by an induction on $n$.
The analysis goes as follows.

For $H^N_n$,
in addition to the property (\ref{eq:HfH}),
there is also an element $\e$ of $H^N_n$ for some $n$ (in fact
$n=N$ and $\e = \f_N$)
such that the matrix $\rho_N(\e )$ is rank=$1$.
It follows that
\beq \label{eq:eHNNee}
\rho_N(\e) \; \rho ( H^N_N ) \; \rho_N(\e )
     \;\; \subseteq \;\;  \kk \rho_N( \e )
\eq
Indeed we have the following
 (`localization property'): For all $n$,
\beq 
      \label{eq:eHe}
\rho_N(\e) \, \rho_N ( H^N_n ) \, \rho_N(\e )
   \;\; = \;\;   
       \underbrace{\rho_N( \e )}_{on \; V^N} \otimes
         \underbrace{\rho_N(H^N_{n-N} )}_{on \; V^{n-N}}
\eq
(cf. (\ref{eq:eHe00})).

We have from (\ref{eq:eHe}) that
$\e H^N_n \e \cong H^N_{n-N}$ and,
since $\e=\f_N$, that
$H^N_n / H^N_n \e H^N_n \cong H^{N-1}_n$.
So by (\ref{eq:LamA1})
the irreducible representations of $H^N_n$ can be determined by
an iteration on $n$ (and $N$).

It is sometimes possible to lift this to the
loop-braid case.  How might the braid group paradigm generalize?
Of course every finite-dimensional quotient of the group algebra of
the braid group $\B_n$
has a local relation --- a polynomial relation $\chi_\qq = 0$ obeyed
by each braid generator.
Thus we can start, organisationally, by fixing such a
relation.
If this relation is quadratic then the quotient algebra is
finite dimensional for all $n$, in particular it is
the Hecke algebra.
If the local relation is cubic or higher order then this quotient
alone is not enough to make the quotient algebra finite-dimensional
for all $n$
(and also not enough to realise the localisation property,
as in \S\ref{ss:loc}).

Below we study 
group-type representations of $LB_n$ in this context.

\subsection{Some more preparations: the BMW algebra}\label{ss:BMW1}

We define the BMW algebra
over $\C$ as follows.
For $n \in \N$ and $r,q \in \C^*$ with $q^2 \neq 1$,
$\C$-algebra
 $\CC_n(r,q)$ is generated by $b_1,b_2,...,b_{n-1}$ and inverses
obeying the braid relations
\beq \label{eq:rB}
\braidrel{b}{i},
\qquad\qquad\qquad
\commrel{b}{i}{j}
\eq
and, defining
\beq \label{eq:u_i}
\hspace{1cm}
u_i = 1 - \frac{b_i - b_i^{-1}}{q-q^{-1}}
  \;\;\; = \; \frac{b_i^{-1}}{q^{-1} - q}(b_i - q)(b_i + q^{-1}) ,
\eq
obeying the additional relations
\beq \label{eq:r1}
u_i b_i = r^{-1} u_i
\eq
\beq \label{eq:r2}
u_i b_{i-1}^{\pm 1} u_i = r^{\pm 1} u_i  .
\eq

Relation (\ref{eq:r1}) is equivalent to a `cubic local relation'
\beq \label{eq:rloc}
(b_i - r^{-1}) (b_i -q) (b_i + q^{-1}) = 0 .
\eq

Relation (\ref{eq:r1}) also implies   
\[
u_i^2 = (1+\frac{r-r^{-1}}{q-q^{-1}}) u_i .
\]
Relation (\ref{eq:r2}) implies
\[
u_i u_{i \pm 1} u_i = u_i .
\]
Of course we also have from (\ref{eq:rB}):
\[
\commrel{u}{i}{j} .
\]
Indeed the $u_i$'s generate a Temperley--Lieb subalgebra of $\CC_n(r,q)$.
This subalgebra realizes a different quotient of the braid group
algebra: the images of the braid generators are
$
a_i \; = \; 1-q'(q,r) \; u_i
$,
where $q'$ is defined by
$q'+{q'}^{-1} = 1+\frac{r-r^{-1}}{q-q^{-1}}$
with a quadratic local relation, and with the two eigenvalues
depending on $q$ and $r$.

For us the interesting case of $\CC(r,q)$ is $r=q$, where the
braid generators of the TL subalgebra
obey the symmetric group relations.
In this case, then, we have images of both the braid group and the
symmetric group in $\CC_n(q,q)$, as for $LB_n$.
Indeed we have the following.

\begin{lemma}
There is a map $\psi : LB_n \rightarrow \CC_n(q,q)$ given by
$\sigma_i \mapsto b_i$, $s_i \mapsto a_i =1-u_i$.
\end{lemma}

\proof
With $r=q$ we have $u_i^2 = 2u_i$ and $q' = 1$ so $a_i = 1-u_i$
and $a_i^2 = 1$ as already noted. Relations (L1,L2) can be directly checked.

\subsection{The representations $\rhoo_N$ of $LB_n$}\label{ss:rhooN}
 { \newcommand{\xx}{x} \newcommand{\xxx}{1}
For each $N$, and $x \in \C$,
there is a well-known local representation $\rhoox_N$ of
$\C[\B_n] / I_{\xx,\xxx,-\xxx}$
(trivially rescalable, setting $x=q^2$,
to a representation $\rhooq_N$ of
$\C[\B_n] / I_{q,q^{-1},-q^{-1}}$;
and that in case $N=2$ is also a representation of $\CC_n(q,q)$).
One takes the diagonal BVS with
\ignore{{
$g_1 = \mat{cc} 1&0 \\ 0&x \tam$;
$g_2 = \mat{cc} x&0 \\ 0&1 \tam$.
}}
$$ {
\arraycolsep=2.5975pt
\def\arraystretch{.92}
g_1 = \mat{cccc} \xx&0 \\ 0&\xxx \\ 0&0&\ddots \\ 0&0&0&\xxx \tam
, \;\;
g_2 = \mat{cccc} \xxx&0 \\ 0&\xx \\ 0&0&\ddots \\ 0&0&0&\xxx \tam
, ..., \;
g_N = \mat{cccc} \xxx&0 \\ 0&\xxx \\ 0&0&\ddots \\ 0&0&0&\xx \tam .
}
$$
We abbreviate the basis element
$e_{i_1} \otimes e_{i_2} \otimes ...\otimes e_{i_n}$ of $V^n$
as $|i_1 i_2 ... i_N\rangle$, so that $e_1 \otimes e_1 \otimes e_2$ becomes
$|112\rangle$ and so on. Then
\beq \label{eq:charge1}
\sigma_j  |i_1 i_2 ... i_N\rangle = \left\{ \begin{array}{ll}
           x  |i_1 i_2 ... i_N\rangle &  i_{j} = i_{j+1} \\
           \;\;\;  |i_1 i_2 ... i_{j+1} i_j ... i_N\rangle & \mbox{ otherwise}
\end{array} \right.
\eq
Specifically for $N=2$ (with basis elements of $V^2$ ordered
$|11\rangle,|12\rangle,|21\rangle,|22\rangle$):
$$
\sigma_i \stackrel{\rhoox_2}{\mapsto}
        \;  \Id_2 \otimes \Id_2 \otimes ... \otimes
\mat{cccc} \xx&0&0&0 \\ 0&0&\xxx&0 \\ 0&\xxx&0&0 \\ 0&0&0&\xx \tam
           \otimes \Id_2 \otimes ... \otimes Id_2
$$
 }
Strictly speaking we need to rescale:
$g_1 = \mat{cc} q&0 \\ 0&1/q \tam$;
$g_2 = \mat{cc} 1/q&0 \\ 0&q \tam$.
So
$$
\sigma_i \stackrel{\rhooq_2}{\mapsto}
    \;  \Id_2 \otimes \Id_2 \otimes ... \otimes
\mat{cccc} q&0&0&0 \\ 0&0&1/q&0 \\ 0&1/q&0&0 \\ 0&0&0&q \tam
           \otimes \Id_2 \otimes ... \otimes Id_2
$$

This gives, for example,
\beq \label{eq:si2}
\frac{\sigma_i -\sigma_i^{-1}}{q-q^{-1}}
   \stackrel{\rhooq_2}{\mapsto}
  \;  \Id_2 \otimes \Id_2 \otimes ... \otimes
\mat{cccc} 1&0&0&0 \\ 0&0&-1&0 \\ 0&-1&0&0 \\ 0&0&0&1 \tam
           \otimes \Id_2 \otimes ... \otimes Id_2
\eq

Let us   
define quotient $\C$-algebra
$$
\Bt_n = \C[\B_n] /\ker \rhoo_N .
$$

{\mpr{ \label{pr:semisimple}
The algebra $\Bt_n$ is semisimple.
}}
\proof In case $x $ is real the algebra is evidently generated by hermitian
(indeed real symmetric)
matrices. In other cases one can show that the same is true for a
different generating set.
\qed


{\mpr{ \label{pr:NLBrep}
The map
$s_i \mapsto \rhoo_N( \frac{\sigma_i -\sigma_i^{-1}}{q-q^{-1}} )$
extends $\rhoo_N$ to a representation of $LB_n$.
That is to say, $\Bt_n$ is a quotient of $\C[LB_n]$.
}}

{\mpr{ \label{pr:2LBrep}
The case $\rhoo_2$ factors through $\CC_n(q,q)$. That is
$b_i \mapsto \rhoo_2( \sigma_i )$ gives a representation of
$\CC_n(q,q)$.
}}

Given any realization of $\B_n$, and $q \in \C$, we
define $u_i$ as in (\ref{eq:u_i}).
(The image $\rhoo_2( u_i)$
 obeys the BMW relation  
(\ref{eq:r2}), but $\rhoo_N(u_i)$ for $N>2$ does not.)
As noted in (\ref{pr:NLBrep}),
$s_i \mapsto  \rhoo_N( a_i = 1-u_i)$
 gives a representation of
$\Sn_n$ for each $N$.
Indeed the $\rhoo_N$ representation of $\Sn_n$  
coincides with the classical case, $q=1$, of the $\rho_N$ Hecke algebra
representation:   
\beq \label{eq:Uu}
\rho_N^{q=1}(U_i) = \rhoo_N(u_i) .
\eq
Thus from \S\ref{ss:-1}
we have  
\newcommand{\fin}{\f^1_N}   
\newcommand{\ff}{\f^1}      
$
\ff_N \; := \; \rho_N^{q=1}(\f_N ) \in \rhoo_N .
$

Given a loop BVS 
one obvious question is: Do we have an analogue of (\ref{eq:eHe})
here together with corresponding strong representation theoretic
consequences?
We are particularly interested in cases that do not factor over
$SLB_n$. But the question is hard in general and it is instructive to
start with a `toy' such as the class of loop BVSs above.

\subsection{Fixed charge submodules of $\rhoo_N$}
\newcommand{\E}{E}  
\newcommand{\Res}{\mbox{Res}\,}


One aim is to decompose the representations $\rhoo_N$ into irreducible
representations. To this end,
note that the subspaces of $\rhoo_N$
of fixed $N$-colour-charge are invariant under the $\C[LB_n]$ action.

{\mlem{
The $\Sn_N$ action permuting the standard ordered
basis $\{ e_1, e_2, ..., e_N \}$
of $V = \C^N$  commutes with the $LB_n$ action on $V^n$.
\qed
}}

We write the action of $\Sn_N$ on the right. So if $M$ is an
$LB_n$-submodule of $V^n$ then $Mw$ is an isomorphic submodule for any
$w \in \Sn_N$. This $\Sn_N$ action acts on the set of charges.
Thus we can
index charge-submodules (up to isomorphism) by
the set $\Lambda_{N,n}$ of integer partitions of
$n$ of maximum depth $N$. This is  the same as the charge
decomposition of the Hecke algebra representation $\rho_N$ (where the
submodules are called {\em Young modules}).
But the further decomposition into irreducibles is not the
same as in the Hecke case.

For an explicit example,
\ignore{{
we abbreviate the basis element
$e_{i_1} \otimes e_{i_2} \otimes ...\otimes e_{i_n}$ of $V^n$
as $i_1 i_2 ... i_N$, so that $e_1 \otimes e_1 \otimes e_2$ becomes
112 and so on.
Then }}%
the basis $B_\lambda$ for the $\lambda$ subspace
in case $\lambda=(2,1)$ is
$ 
B_{21} \; = \; \{ 112, 121, 211 \} .
$ 
We write $Y_\lambda$ for the charge $\lambda$
submodule.
Thus we have
\beq \label{eq:young1}
\rhoo_{N,n} \cong \; \bigoplus_{\lambda \in \Lambda_{N,n}} m_\lambda Y_\lambda
\eq
where $m_\lambda$ is the multiplicity.

If $\Sn_N$ or a subgroup $G$ fixes a submodule $Y_\lambda$ then
this module is itself a right $G$-module and
an idempotent
decomposition of 1 in $\C[G]$ induces a decomposition of $Y_\lambda$.

For each $\lambda$ there is a $G$ fixing $Y_\lambda$, call it
$G_\lambda$, a Young subgroup of $\Sn_N$ (possibly trivial). As usual
an idempotent decomposition of $1$ in $\C[G]$ may be characterised by
tuples of Young diagrams/integer partitions. For each such label
there is also a secondary index running over the dimension of the corresponding
irreducible representation of $G$; but idempotents with the same
primary label are isomorphic.
If $Y_\lambda$ has a non-trivial such decomposition we will write
$Y_{\lambda}^{\underline{\mu}} $ for the submodule with primary label
$\underline{\mu}$.
We call these modules $Y_{\lambda}^{\underline{\mu}} $ {\em harmonic
  modules}.
For given $\lambda$ write $\Lambda_\lambda$ for the set of primary
labels
(the index set for irreducible representations
 $\Delta_{\underline{\mu}}$ of $G_\lambda$). Thus
\beq \label{eq:harmonicX}
Y_\lambda \; = \; \bigoplus_{\underline{\mu} \in \Lambda_\lambda }
                      \dim \Delta_{\underline{\mu}} Y_{\lambda}^{\underline{\mu}}
\eq

Note that the decomposition of $Y_\lambda$ into irreducible modules
for the restriction to the `classical' subalgebra $H^N$ generated by the
$u_i$s (the symmetric group action)
is well-known.
This gives a lower bound on the size of summands of
$Y_\lambda$ as a module for the full algebra.

{\mlem{ \label{lem:same1n}
Actions of subgroup $\Sn_n$ and $\B_n$ on $Y_{(1^n)}$ are identical up to
sign.  \qed
}}

Comparing the `classical' decomposition of $Y_{(1^n)}$ above with the
idempotent decomposition with $G= \Sn_N = \Sn_n$ in this case we see that
they are the same.

To apply localization
later we will be interested, for each given $N$,
in detecting submodules $M$ of $Y_\lambda$ on
which $\e=\ff_N$ acts like 0.
We call these $e$-null, or $\ff_N$-null, submodules.
Any such submodule $M$ decomposes also as an $\Sn_n$-submodule, and so
$\ff_N$ would have to act like 0 on each of the submodules in this
decomposition. For example in case $N=2$ only the
irreducible $\Sn_n$-module
$\Delta_{(n)}$ has this property at rank-$n$.
So here there can only be such a submodule $M$ if $\Delta_{(n)}$ is also an
$LB_n$-submodule of $Y_\lambda$.
A basis element for $\Sn_n$-submodule $\Delta_{(n)}$ in $Y_\lambda$ is known.
We take
\[
b = \sum_{w \in \Sn_n}  w \; 111...222
\]
where  
$111...222$ is
the initial basis element of $Y_\lambda$ in the lex order.
Then for example
$
b \stackrel{\lambda=(2,1)}{=} 2(112 - 121 + 211)
$.
Note here that
$
q \sigma_1 (112 -121 +211) = x112 -211 +121
$,
so $\Delta_{(3)}$ is not an $LB_3$-submodule unless $x=-1$.

{\mlem{ \label{lem:enull1}
(I) In case $N=2$, $x \neq -1$,
no $Y_\lambda$ has  $e$-null proper submodule except in case
$\lambda=(1,1)$, where $Y_{(1^2)}^{(1^2)}$ is $\ff_2$-null.
\\
(II) In case $N=3$ 
the module $Y_{(n-2,1^2)}^{(1^2)}$ is
$\ff_3$-null for $n>3$.
}}
\proof (I) The example above is indicative, except in case $(1,1)$ where
there is no $x$ term.
\\ (II) A basis of $Y_{(2,1^2)}^{(1^2)}$ is $\{ 1123-1132,
1213-1312, 1231-1321, 2113-3112, 2131-3121, 2311-3211  \}$.
One readily checks the $\ff_3$ action on this. The other cases are similar.
\qed

There is an injective algebra map
$$
\Bt_{n-N} \stackrel{\sim}{\rightarrow} \ff_N \otimes \Bt_{n-N}
 \hookrightarrow \ff_N \Bt_n \ff_N  .
$$
Thus any $\Bt_n$-module gives rise to a $\Bt_{n-N}$ module by first
localizing (we will write simply $F$ for the localisation functor
$F_{\ff_N}$ here),
then restricting.
{\mlem{ \label{lem:fY}
There is an isomorphism of $\Bt_{n-N}$-modules
\[
\ff_N Y_{(\lambda_1, \lambda_2, ..., \lambda_N)} \; \cong \;
    \left\{ \begin{array}{ll}
     Y_{(\lambda_{1} -1, \lambda_2 -1, ..., \lambda_N -1)}   &
           \lambda_N > 0
\\
0 & \lambda_N = 0
\end{array} \right.
\]
}}
\proof{}

\newcommand{\otN}{12...N}

For any given $N$ we can  
write $w \in B_{\lambda}$
as
$$
w = \;
 \underbrace{w_1 w_2 ...w_N}_{w_-} \underbrace{ w_{N+1} w_{N+2} ... w_n }_{w_+}
= \; w_- w_+  .
$$
Then
\newcommand{\Ef}{\ff_N} 
\beq \label{eq:E0}
\Ef w = \Ef w_- w_+ = \left\{ \begin{array}{ll}
         0   & \mbox{  unless $w_-$ is a permutation of \otN. }
\\
     \overline{\otN} w_+ & \mbox{  $w_-$ is a perm. of \otN. }
\end{array} \right.
\eq
where $\overline{123}=123+213+132+231+312+321$ and so on.
That is, $\Ef V^{n} \cong V^{n-N}  $ and $\Ef Y_\lambda \cong Y_{\lambda-(1^N)}$
as vector spaces, and hence modules. \qed




{\mlem{ \label{lem:fYmu}
Let $\lambda \in \Lambda$ and $l=l_\lambda$ the number of distinct
row-lengths in $\lambda$, so that $\underline{\mu}$ in
$Y_{\lambda}^{\underline{\mu}}$
has $l_\lambda$ distinct components (each $\mu_i$ a partition).
Let  $\underline{\mu}'$ denote $\underline{\mu}$ with the $l$-th
component omitted.
There is an isomorphism of $\Bt_{n-N}$-modules
\[
\ff_N Y_{(\lambda_1, \lambda_2, ..., \lambda_N)}^{\underline{\mu}} \; \cong \;
    \left\{ \begin{array}{ll}
     Y_{(\lambda_{1} -1, \lambda_2 -1, ..., \lambda_N -1)}^{\underline{\mu}}   &
           \lambda_N > 1
\\
     Y_{(\lambda_{1} -1, \lambda_2 -1, ..., \lambda_N -1)}^{\underline{\mu}'}   &
           \lambda_N = 1, \; (\mu_l)_2 =0
\\
0 & \lambda_N = 1, \; (\mu_l)_2 >0
\\
0 & \lambda_N = 0
\end{array} \right.
\]
}}
\proof{The decomposition of $Y_\lambda$ by the right-action of the
  charge group,
commutes with the left-action of $\ff_N$. So, noting
  Lemma~\ref{lem:fY}, it only remains to verify the $\lambda_N =1$
  cases.
In these cases the first column of $\lambda$ is uniquely the longest,
of length $N$. Thus the colours involved in the last component of
$\underline{\mu}$ are symmetrised by $\ff_N$.
Any colour symmetry idempotent acting from the right corresponding to
$\mu_l$ with $(\mu_l)_2 >0$ involves an antisymmetriser in its
construction,
and hence annihilates $\ff_N Y_\lambda$.
\qed}


{\mlem{ \label{pr:harmnoniso}
For $x \neq -1$
the harmonic modules of $LB_n$, i.e. the modules
$\{ Y_{\lambda}^{\underline{\mu}} \; | \; \lambda \in \Lambda_n, \;
          \underline{\mu} \in \Lambda_\lambda \}$,
are pairwise non-isomorphic.
}}
\proof Work by induction on $n$.
Compare $Y=Y_{\lambda}^{\underline{\mu}}$,
$Y'= Y_{\lambda'}^{\underline{\mu}'}$, say, with
$\underline{\mu}  \neq \underline{\mu}'$.
If either $\ff_N Y$ or $\ff_N Y' \neq 0$
for some $N \geq || \lambda || := \lambda^t_1$
then $Y \not\cong Y'$ by
  Lemma~\ref{lem:fYmu} and the inductive assumption.
The remaining cases
are when one or both of $Y,Y'$ are of type-III in Lemma~\ref{lem:fYmu}.
These are routine to check.
 \qed

How can we understand this proliferation of submodules?
Analogous results to the above hold for the Hecke quotients
$H^N_n$. There it is very useful to use a geometrical principle to
organise the indexing sets for canonical classes of modules (such as
Young modules; or simple modules --- except that there it turns out
that, roughly speaking, the same index set can be used for these
different classes). One way to understand this geometry comes from the
theory of weight spaces in algebraic Lie theory.
Here we do not have any such dual picture, but we can naively bring
over the same organisational principle. This tells us to consider
$\lambda$ as a vector in $\R^N$, and then to draw the set of
$\lambda$s in $\R^{N-1}$ by projecting down the $(1,1,...,1)$
line. One merit of this is that it allows us to draw the entire $N=3$
`weight space' of Young module indices in the plane.

\subsection{Branching rules for harmonic modules}
We consider here the natural restriction from $LB_n$ to $LB_{n-1}$, and claim Fig.\ref{fig:sl3123} gives the branching rules for $N=3$.

{\mpr{ \label{pr:Ybranch}
The branching rules for Young modules corresponding to the natural restriction
from $LB_n$ to $LB_{n-1}$ are
\[
\downarrow Y_\lambda = \bigoplus_i Y_{\lambda - e_i}
\]
where the sum is over removable boxes in the Young diagram $\lambda$.
}}
\proof The $LB_{n-1}$ action ignores the last symbol in the
colour-word basis for $Y_\lambda$.
\qed

{\mpr{ \label{pr:branching3}
The directed graph in Fig.\ref{fig:sl3123}
gives the branching rules for harmonic modules for $N=3$,
 using the geometric realisation.
}}
\proof First note that well in the interior of the picture
the Young and harmonic modules coincide and we can use Prop.\ref{pr:Ybranch}.
Specifically this gives all cases in the forward cone of the point
$(4,2)$.

The remaining cases in the forward cone of $(2,1)$ may be verified by
using Propositions~({\ref{pr:Ybranch}}) and (\ref{eq:harmonicX}).

For the remaining `boundary' cases we split up into cases in
the following subsets: \\
(ii)  the $(1,0)$-ray of point $(2,0)$; \hspace{1in}
(iv) the $(1,1)$-ray of point $(2,2)$; \\
(vi) the point $(1,0)$; \hspace{1cm}
(vii) the point $(1,1)$; \hspace{1cm}
(viii) the point $0$.

We indicate the proof  
with two representative examples.
Case
(ii): In the fibre over $(2,0)$ we have
\[
\downarrow Y_{(6,4,4)}^{(2)} = \; Y_{(6,4,3)} \oplus Y_{(5,4,4)}^{(2)}
\]
by using \ref{pr:Ybranch} and commutativity of (left) restriction with
the (right) idempotent decomposition.
\\
Case (vi): In the fibre over $(1,0)$ we have
\[
\downarrow Y_{(2,1,1)}^{((1),(2))} = Y_{(2,1)} \oplus
       Y_{(1^3)}^{((3))} \oplus Y_{(1^3)}^{((2,1))}
\]
Here the basis is
\[
1123+1132, \; 1213+1312, \; 1231+1321, \;
2113+3112, \; 2131+3121, \; 2311+3211
\]
It will be clear that on restriction the 1st, 2nd and 4th give a basis
of $Y_{(2,1)}$, while the remainder inject into $Y_{(1^3)}$,
and indeed into $Y_{(1^3)} (1+s_2)$.
\qed

\begin{figure}
\[
\includegraphics[width=9.7cm]{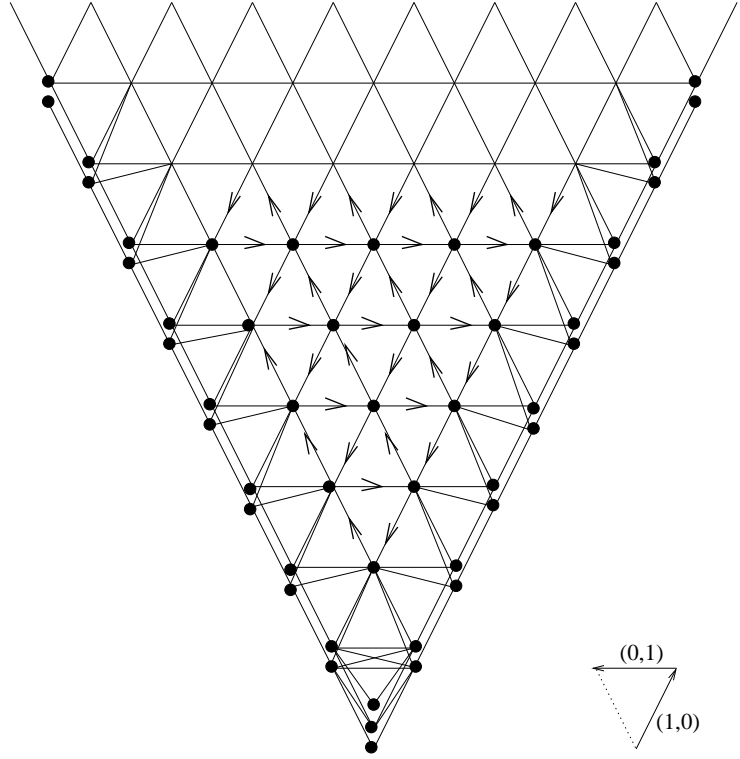}
\]
\caption{Branching rules for harmonic modules
for $N=3$.
All `parallel' edges are directed in the same direction.
\label{fig:sl3123}}
\end{figure}


\begin{theorem}
In cases $N=2,3$, $x \neq -1$, the harmonic modules are irreducible.
\end{theorem}
\proof
We work by induction on $n$.
Consider a harmonic module $Y$ at level $n$.
By Propositions \ref{pr:branching3}, \ref{pr:harmnoniso}
and the inductive assumption
restriction to $n-1$ is multiplicity-free. So
it is enough to show that there is a basis element $b$ in a good basis
with respect to this restriction
(a basis that decomposes into bases for the summands of the restriction)
such that $Y=\Bt_n b$.

In case $Y$ is also a Young module it is easy to see that
$Y=\Bt_n b$ for any standard basis element;
and that the standard basis is a good basis for the restriction to
Young modules; and that at least one of these is a summand of the
restriction to harmonic modules.

In case $Y$ is not a Young module (i.e. on the boundary)
the modification is routine and
we content ourselves here with some representative examples:
\\
(1) Recall that
$\Res Y_{(2,1,1)}^{(2)} = \; Y_{(1,1,1)}^{(3)} \oplus Y_{(1,1,1)}^{(2,1)}
                          \oplus Y_{(2,1)}^{}  .
$
An element lying in the last summand is $1213 + 1312$.
Acting with $\sigma_3$ on this we get $1231 + 1321$.
It is easy to see that this generates the whole module.
\\
(2) Recall that
$\Res Y_{(2,2,1)}^{(2)} = \; Y_{(2,1,1)}^{(2)} \oplus Y_{(2,1,1)}^{(1^2)}
                          \oplus Y_{(2,2)}^{(2)}  .
$
An element lying in the last summand is $11223 + 22113$.
Acting with $\sigma_4$ on this we get $11232 + 22131$.
It is easy to see that this generates the whole module.
\\
(3) We have
$$ \Res Y_{(4,2,2)}^{(2)} = Y_{(4,2,1)} \oplus Y_{(3,2,2)}^{(2)}  . $$
A good basis is
$\{ 11112233+11113322, \; 11112323+11113232, \; ..., \;
11123213+11132312, $
$ \; ..., \; 32211113+23311112 , \; ... \}$,
where all the explicitly written elements lie in the basis for
$Y_{(4,2,1)}$ in the restriction (the first word ends in 3).
Now apply $\sigma_7$: $\sigma_7 (11123213+11132312) = 11123231+11132321$,
which lies in $Y_{(3,2,2)}^{(2)}$.
\qed

Categorical versions of
 the structure for  $N=2$ and $N=3$   
also can be worked out explicitly.
(But in light of Proposition~\ref{pr:semisimple}
these are not as powerful a tool here as
in the corresponding Hecke cases.)
We will leave them  
for future publication.

We make the obvious conjecture for the generalisation to
higher $N$: that the harmonic modules are again a complete set of
irreducibles. 



\end{document}